\documentclass[12pt]{amsart}

\usepackage{hyperref}
\usepackage{amssymb, latexsym, amsthm, enumerate, mathptmx}

\usepackage{tikz}
\usetikzlibrary{arrows}
\usepackage[all]{xy}
\usepackage{color}
\usepackage{cleveref}
\usepackage{mathtools}
\usepackage{paralist}
\usepackage{comment}

\newtheorem{theorem}{Theorem}[section]
\newtheorem{lemma}[theorem]{Lemma}
\newtheorem{corollary}[theorem]{Corollary}

\newtheorem*{claim}{Claim}

\newtheorem{conjecture}[theorem]{Conjecture}
\newtheorem{problem}[theorem]{Problem}

\theoremstyle{definition} % ’here we change the style’
\newtheorem{definition}[theorem]{Definition} % numbered with thm

\theoremstyle{remark} % ’style changed again’
\newtheorem{remark}[theorem]{Remark}
\newtheorem{example}[theorem]{Example}

\DeclareMathOperator{\Inc}{Inc}

\DeclareMathOperator{\Sq}{S}
\DeclareMathOperator{\B}{B}

\newcommand{\N}{\mathbb{N}}
\newcommand{\Z}{\mathbb{Z}}
\newcommand{\ub}{\mathbf{u}}
\newcommand{\vb}{\mathbf{v}}
\newcommand{\wb}{\mathbf{w}}

\newcommand{\cinv}{combinatorial $\Inc$-invariant}

\title{A Kruskal-Katona type result and applications}

\author{Dinh Van Le}
\address{Institut f\"ur Mathematik, Universit\"at Osnabr\"uck, 49069 Osnabr\"uck, Germany}
\email{dlevan@uos.de}

\author{Tim R\"{o}mer}
\address{Institut f\"ur Mathematik, Universit\"at Osnabr\"uck, 49069 Osnabr\"uck, Germany}
\email{troemer@uni-osnabrueck.de}

\begin{document}

\begin{abstract}
Inspired by the Kruskal-Katona theorem a minimization problem is studied, where the role of the shadow is replaced by the image of the action of the monoid of increasing functions. One of our main results shows that compressed sets are a solution to this problem. Several applications to simplicial complexes are discussed.
\end{abstract}
\maketitle
\section{Introduction}
The celebrated Kruskal-Katona theorem (independently discovered by Sch\"{u}tzenberger \cite{Sc}, Kruskal \cite{Kr}, Katona \cite{Kat}, Harper \cite{Ha}, and Clements--Lindstr\"{o}m \cite{CL}) solves the \emph{shadow minimization problem}: for a finite family $\mathcal{F}\subseteq \binom{\N}{d}$ of $d$-sets of positive integers, the \emph{shadow} $\partial\mathcal{F}$ of $\mathcal{F}$ is of minimal possible size if $\mathcal{F}$ is a \emph{compressed} set (see Section \ref{sec2} for definitions and more details).

This result has many important consequences, generalizations, and related results; see, e.g., \cite{An, GK,HH11}. For example, the numerical version of the theorem yields a characterization of possible face numbers of simplicial complexes, whereas its algebraic version characterizes possible Hilbert functions of standard graded exterior algebras over a field \cite[Chapter 6]{HH11}. Note that possible Hilbert functions of standard graded commutative algebras over a field are described by Macaulay's classical theorem \cite{Ma}. Among generalizations of the Kruskal-Katona theorem, perhaps the most prominent one was given by Clements and Lindstr\"{o}m \cite{CL}, who extended the result to multisets (see, in particular, \cite[Theorem 9.1.1, Corollary 9.2.3]{An} for a shadow and a so-called shade version). In fact, Clements-Lindstr\"{o}m's result also contains Macaulay's one as a special case.

Inspired by these theorems, we study a similar minimization problem,
where the role of $\partial\mathcal{F}$ is replaced by the image of
$\mathcal{F}$ under a monoid action of interest. More precisely, consider the monoid of increasing functions on $\N$:
\[
\Inc = \{\pi \colon \N \to \N \mid \pi(j)<\pi(j+1)\ \text{ for all }\ j\ge 1\}
\]
and the following subset of it:
\[
\Inc_1= \{\pi \in \Inc \mid \pi(j)\le j+1\ \text{ for all }\ j\ge 1\}.
\]
We write a $d$-set $\ub\in\binom{\N}{d}$ in the form $\ub=(u_1,\ldots,u_d)$ with $u_1<\cdots <u_d$ and we let a function $\pi\in\Inc_1$ act on $\ub$ as follows
\begin{equation}
\label{action}
 \pi(\ub)=(\pi(u_1),\ldots,\pi(u_d))\in \binom{\N}{d}.
\end{equation}
Now define the \emph{$\Inc$-image} of a family $\mathcal{F}\subseteq \binom{\N}{d}$ to be
\[
 \Inc(\mathcal{F})=\{\pi(\ub)\mid \ub\in \mathcal{F},\ \pi \in \Inc_1\}\subseteq \binom{\N}{d}.
\]
Observe that $\Inc(\mathcal{F})$ is a finite set if $\mathcal{F}$ has this property (see \Cref{sec3}). We are then interested in the following minimization problem:
\begin{problem}
\label{pb}
Among all finite families $\mathcal{F}$ of $d$-subsets of $\N$ with $|\mathcal{F}|$ fixed, find a family with minimal $|\Inc(\mathcal{F})|$.
\end{problem}

It is worth noting that the monoid $\Inc$ plays a crucial role in the study of ideals in infinite dimensional polynomial rings that are invariant under the action of the infinite symmetric group; see, e.g., \cite{AH07,Dr14,HS12,LNNR,LNNR2,NR17,NR17+}. Such ideals have recently attracted great attention because they arise naturally in various areas of mathematics, including
algebraic chemistry \cite{Dr10}, algebraic statistics \cite{BD11,DEKL15,DK14,HM13,HS07,SS03},
group theory \cite{Co67}, and
representation theory \cite{CEF, PS, SS-14,SS-16}. Moreover, modules with $\Inc$-action are interesting on their own right. Recently, the representation theory of these modules has been studied extensively in \cite{GS}.
Another motivation for studying \Cref{pb} is described in the following. For any family $\mathcal{F}$ of subsets of $\N$ define
\[
 \Inc(\mathcal{F})=\bigcup_{d\ge 1}\Inc(\mathcal{F}_d),
\]
where $\mathcal{F}_d$ is the subset  of $\mathcal{F}$ consisting of all $d$-sets. For each $n\ge 1$ let $\Delta_n$ be a simplicial complex on the vertex set $[n]:=\{1,\dots,n\}$. Denote by $\mathcal{N}(\Delta_n)$ the set of non-faces of $\Delta_n$.
Then a chain of simplicial complexes $(\Delta_n)_{n\ge 1}$ is called \emph{$\Inc$-invariant} if
\[
\Inc(\mathcal{N}(\Delta_n))\subseteq \mathcal{N}(\Delta_{n+1})\quad\text{for all }\ n\ge 1.
\]
Equivalently, this amounts to saying that the chain of Stanley-Reisner ideals of the $\Delta_n$ is invariant under the action of the monoid $\Inc$. See, e.g., \cite{NR17} for definitions and more information on $\Inc$-invariant chains of ideals.
The study of Hilbert series and asymptotic behavior of $\Inc$-invariant chains of ideals stimulates the following problem (for more details on shifting theory, the reader is referred to \cite[Chapter 11]{HH11} and \cite{Ka}):

\begin{problem}
 Find a shifting operation $\mathcal{S}$ such that for any $\Inc$-invariant chain of simplicial complex $(\Delta_n)_{n\ge 1}$, the chain $(\mathcal{S}(\Delta_n))_{n\ge 1}$ is also $\Inc$-invariant.
\end{problem}

In \cite[Conjecture 7.2]{LNNR}, it is conjectured that the symmetric shifting operation is a solution to this problem. Based on computational evidence, we predict that the exterior shifting operation provides another solution:

\begin{conjecture}
 \label{conj}
 Let $(\Delta_n)_{n\ge 1}$ be an $\Inc$-invariant chain of simplicial complexes. If $\mathcal{S}$ is either the symmetric or exterior shifting operation, then $(\mathcal{S}(\Delta_n))_{n\ge 1}$ is also $\Inc$-invariant.
\end{conjecture}

If $(\Delta_n)_{n\ge 1}$ is an $\Inc$-invariant chain, then it \emph{stabilizes}, i.e. $\Inc(\mathcal{N}(\Delta_n))= \mathcal{N}(\Delta_{n+1})$ for $n\gg0$; see \cite[Theorem 3.6]{HS12}. Note that for any shifting operation $\mathcal{S}$ it holds that $|\mathcal{N}(\Delta_n)_d|=|\mathcal{N}(\mathcal{S}(\Delta_n))_d|$ for all $d\ge 1$. So if \Cref{conj} is true, then we have
\begin{align*}
 |\Inc(\mathcal{N}(\mathcal{S}(\Delta_n)_d)|\le |\mathcal{N}(\mathcal{S}(\Delta_{n+1}))_d|=|\mathcal{N}(\Delta_{n+1})_d|
 =|\Inc(\mathcal{N}(\Delta_n)_d)|
\end{align*}
for all $n\gg0$ and $d\ge 1$. Thus,
it is natural to ask when $|\Inc(\mathcal{N}(\Delta_n)_d)|$ is minimal given that $|\mathcal{N}(\Delta_n)_d|$ is fixed, which leads directly to \Cref{pb}.

In the spirit of the Kruskal-Katona theorem, our first main result solves \Cref{pb}:

\begin{theorem}
\label{main}
Let $\mathcal{F}\subseteq \binom{\N}{d}$ be a finite compressed set. Then for any $\mathcal{F}'\subseteq \binom{\N}{d}$ with the property $|\mathcal{F}'|=|\mathcal{F}|$ one has $|\Inc(\mathcal{F}')|\ge|\Inc(\mathcal{F})|$.
\end{theorem}

Our proof is based on a variation of the compression technique established for multisets by Clements and Lindstr\"{o}m \cite{CL} (which stems from the idea of compression introduced by Lindstr\"{o}m and Zetterstr\"{o}m \cite{LZ} in another context). Note that the combinatorial shifting technique is not applicable to prove this theorem; see \Cref{re}.

\Cref{main} leads to some interesting consequences about simplices complexes. In particular, as the second main result we obtain in \Cref{thm:onemainresult} a characterization of possible $f$-vectors of chains of simplicial complexes that are \emph{combinatorial} invariant under the action of the monoid $\Inc$ (see \Cref{sec4} for the definition of this notion).

The paper is organized as follows. In \Cref{sec2} we recall the Kruskal-Katona theorem and some necessary notions from extremal finite set theory. \Cref{main} and its numerical version (\Cref{numerical}) are proved in \Cref{sec3}. 
In \Cref{sec4}, the results in \Cref{sec3} are applied to simplicial complexes. 
Finally, some open problems are proposed in \Cref{sec5}.

%%%%%%%%%%%%%%%%%%%%%%%%%%%%%%%%%%%%%%%%%%%%
\section{Preliminaries}
\label{sec2}

In this section we collect some basic notions and facts in the theory of finite sets, referring the reader to \cite{An} for more details. Let $d$ be a positive integer. Denote by $\binom{\N}{d}$ the set of all $d$-subsets of $\N$. We write each element $\ub\in\binom{\N}{d}$ in the form $\ub=(u_1,\ldots,u_d)$ with $u_1<\cdots <u_d$.
The \emph{squashed order} on $\binom{\N}{d}$ is defined as follows: For $\ub,\vb\in\binom{\N}{d}$ we set $\ub<_{\Sq}\vb$ (or $\ub<\vb$ for short) if the largest element of the symmetric difference $(\ub\backslash\vb)\cup(\vb\backslash\ub)$ belongs to $\vb$. So for instance, some smallest $3$-subsets of $\N$ in this order are
\[
  (1,2,3)<(1,2,4)<(1,3,4)<(2,3,4)<(1,2,5)<(1,3,5)
   <(2,3,5)<\cdots.
\]
Observe that $<_{\Sq}$ is a well-ordering on $\binom{\N}{d}$. In the sequel, unless otherwise stated, $\binom{\N}{d}$ will be ordered by $<_{\Sq}$.

A finite family $\mathcal{F}\subseteq \binom{\N}{d}$ is said to be {\em compressed} if $\mathcal{F}$ is an initial segment of $\binom{\N}{d}$ with respect to $<_S$, i.e., $\mathcal{F}$ consists of the $|\mathcal{F}|$ smallest elements of $\binom{\N}{d}$.
For $\ub\in \binom{\N}{d}$ set
$$C(\ub)=\Big\{\vb\in \binom{\N}{d}\mid \vb\le \ub\Big\}.$$
It is apparent that $C(\ub)$ is a finite compressed set. Conversely, if $\mathcal{F}\subseteq \binom{\N}{d}$ is a finite compressed set, then $\mathcal{F}=C(\ub)$ with $\ub=\max \mathcal{F}$.

For any finite family $\mathcal{F}\subseteq \binom{\N}{d}$ there exists a unique compressed family $\mathcal{C}(\mathcal{F})\subseteq \binom{\N}{d}$ such that $|\mathcal{C}(\mathcal{F})|=|\mathcal{F}|$. We call $\mathcal{C}(\mathcal{F})$ the {\em compression} of $\mathcal{F}$. More generally, let $\mathcal{A}\subseteq \binom{\N}{d}$ and let $\mathcal{F}\subseteq\mathcal{A}$ be a finite subset. Then one can define the {\em compression} $\mathcal{C}_{\mathcal{A}}(\mathcal{F})$ of $\mathcal{F}$ \emph{in $\mathcal{A}$} to be the subset consisting of the $|\mathcal{F}|$ smallest elements of $\mathcal{A}$ (in the squashed order). In particular, for any $k\in\N$ we will use the notation $\mathcal{C}_{>k}(\mathcal{F})$ to denote the compression of $\mathcal{F}$ in $\binom{\N_{>k}}{d}$, where $\N_{>k}$ is the set of integers bigger than $k$.

The Kruskal-Katona theorem asserts that
\[
 |\partial\mathcal{F}|\ge|\partial\mathcal{C}(\mathcal{F})|\quad\text{for every finite family }\ \mathcal{F}\subseteq \binom{\N}{d},
\]
where the \emph{shadow} $\partial\mathcal{F}$ of $\mathcal{F}$ is defined as
\[
 \partial\mathcal{F}=\{\ub\setminus\{i\}\mid \ub\in \mathcal{F},\ i\in\ub\}\subseteq \binom{\N}{d-1}.
\]
Note that $\partial\mathcal{C}(\mathcal{F})$ is a compressed family; see, e.g., \cite[Theorem 7.5.1]{An}. So the Kruskal-Katona theorem in fact implies a stronger result
\[
 \partial\mathcal{C}(\mathcal{F})\subseteq\mathcal{C}(\partial\mathcal{F})\quad\text{for every finite family }\ \mathcal{F}\subseteq \binom{\N}{d}.
\]

Let us now recall a numerical version of the Kruskal-Katona theorem. Given a positive integer $d$, it is known that any positive integer $m$ has a unique \emph{$d$-binomial representation}:
\begin{equation}
\label{binom}
 m=\binom{a_d}{d}+\binom{a_{d-1}}{d-1}+\cdots+\binom{a_s}{s},
\end{equation}
where $a_d>a_{d-1}>\cdots>a_s\ge s\ge 1$; see, e.g., \cite[Theorem 7.2.1]{An}. Using this expansion we define
\[
 \partial_d(m)=\binom{a_d}{d-1}+\binom{a_{d-1}}{d-2}+\cdots+\binom{a_s}{s-1}.
\]
Also let $\partial_d(0)=0$. Then one can show that for any finite compressed family $\mathcal{F}\subseteq \binom{\N}{d}$ it holds that
\[
 |\partial\mathcal{F}|=\partial_d(|\mathcal{F}|)
\]
(see, e.g., \cite[p. 119]{An}). So the Kruskal-Katona theorem can be equivalently stated as follows: If $\mathcal{F}\subseteq \binom{\N}{d}$ is a finite family, then
\[
 |\partial\mathcal{F}|\ge\partial_d(|\mathcal{F}|).
\]
Note that this version of the theorem yields the well-known characterization of all possible $f$-vectors of simplicial complexes (see \Cref{sec4}).

We conclude this section with the notion of shifted families. Consider the {\em Borel order} $\le_{\B}$ on $\binom{\N}{d}$ defined by
$$\ub\le_{\B} \vb\quad\text{if}\quad u_i\le v_i\quad \text{for all }\ i=1,\ldots,d.$$
For $\ub\in\binom{\N}{d}$ set
$$B(\ub)=\Big\{\vb\in \binom{\N}{d}\mid \vb\le_{\B} \ub\Big\}.$$
A family $\mathcal{F}\subseteq \binom{\N}{d}$ is said to be {\em shifted} if $\vb\in \mathcal{F}$ whenever there exists $\ub\in \mathcal{F}$ with $\vb\le_{\B} \ub$. It is clear that $B(\ub)$ is shifted. More generally, $\bigcup_{i=1}^sB(\ub_i)$ is shifted for any $\ub_1,\ldots,\ub_s$ in $\binom{\N}{d}$. Conversely, if $\mathcal{F}\subseteq \binom{\N}{d}$ is a finite shifted family, then $\mathcal{F}=\bigcup_{i=1}^sB(\ub_i)$, where $\ub_1,\ldots,\ub_s$ are the maximal elements of $\mathcal{F}$.
Evidently, if $\ub\le_{\B} \vb$, then $\ub\le_{\Sq} \vb$. Hence every compressed family is shifted. The converse is of course not always true. For example, the family $\mathcal{F}=\{(1,2),(1,3),(1,4)\}$ is shifted but not compressed in $\binom{\N}{2}$.

%%%%%%%%%%%%%%%%%%%%%%%%%%%%%%%%%%%%%%%%%%%%%%%%%
\section{Proof of Theorem 1.4}%\Cref{main}}
\label{sec3}

This section is devoted to the proof of \Cref{main}.
First, we compute the $\Inc$-image of an element. Note that if $\pi\in \Inc_1$, then there exists $i\in \N\cup\{\infty\}$ such that $\pi=\pi_i$, where
\[
 \pi_i(j)=
 \begin{cases}
  j&\text{for } 1\le j\le i-1,\\
  j+1&\text{for } j\ge i.
 \end{cases}
\]
Now let $\ub=(u_1,u_2,\ldots,u_d)\in \binom{\N}{d}$. It is clear that $\pi_i(\ub)=\ub$ for all $i\ge u_d+1$. Thus,
\begin{equation}
 \label{Inc}
 \begin{aligned}
   \Inc(\ub)=\{&\pi(\ub)\mid \pi \in \Inc_1\}=\bigcup_{i=1}^{u_d+1}\{\pi_i(\ub)\}\\
   =\{&\ub,\;(u_1,u_2,\ldots,u_d+1),\;(u_1,u_2,\ldots,u_{d-1}+1,u_d+1),\ldots,\\
   & (u_1,u_2+1,\ldots,u_d+1),\;(u_1+1,u_2+1,\ldots,u_d+1)\}.
  \end{aligned}
\end{equation}
It follows that for any finite family $\mathcal{F}\subseteq \binom{\N}{d}$, its $\Inc$-image
is also a finite family, since
\[
 \Inc(\mathcal{F})=\bigcup_{\ub\in\mathcal{F}}\Inc(\ub)
\]

\begin{example}
 Let $\mathcal{F}=\{(1,2,4),(1,3,5)\}$. Then
 \[
  \Inc(\mathcal{F})=\mathcal{F}\cup\{(1,2,5),\;(2,3,5),\;(1,3,6),\;(1,4,6),\;(2,4,6)\}.
 \]
\end{example}

Next, we show that compressedness and shiftedness are preserved under the $\Inc$-action. For $\ub=(u_1,u_2,\ldots,u_d)\in \binom{\N}{d}$ set
%\[
 $\ub+\mathbf{1}=\pi_1(\ub)=(u_1+1,u_2+1,\ldots,u_d+1)$.
%\]

\begin{lemma}\label{lm33}
Let $\ub\in \binom{\N}{d}$. Then
$$\Inc(C(\ub))=C(\ub+\mathbf{1}).$$
\end{lemma}

\begin{proof}
	Evidently, $\vb\le \ub+\mathbf{1}$ for all $\vb\in\Inc(\ub)$. It follows that $\Inc(C(\ub))\subseteq C(\ub+\mathbf{1}).$ For the reverse inclusion, suppose $\ub=(u_1,u_2,\ldots,u_d)$ and let  $\vb=(v_1,v_2,\ldots,v_d)\in C(\ub+\mathbf{1})$. We need to show that there exists $\vb'\in C(\ub)$ such that $\vb\in \Inc(\vb')$. Consider the following cases:
	
	{\em Case 1}: $v_1>1$. Then
	\[
	 \vb'=\vb-\mathbf{1}:=(v_1-1,v_2-1,\ldots,v_d-1)\in\binom{\N}{d}.
	\]
	Moreover, one has $\vb'\le\ub$ since $\vb\le\ub+\mathbf{1}$. Thus, $\vb'\in C(\ub)$. Now $\vb=\pi_1(\vb')\in \Inc(\vb')$ and we are done in this case.
	
	{\em Case 2}: $v_1=1$ and $v_{k}-1=v_{k-1}$ for all $k=1,\ldots,d$. Then $v_k=k$ for all $k=1,\ldots,d$. Since obviously $k\le u_k$ for all $k=1,\ldots,d$, one gets $\vb\le\ub$, i.e.,  $\vb\in C(\ub)$. So we may choose $\vb'=\vb$, and then
	$\vb=\pi_{d+1}(\vb)=\pi_{d+1}(\vb')\in \Inc(\vb')$.

	{\em Case 3}: $v_1=1$ and there exists some $k$ such that $v_{k}-1>v_{k-1}$. Let $i$ be the smallest such $k$. Then $v_k=k$ for $k=1,\ldots,i-1$, $v_k>k$ for $k\ge i$, and
	$$\vb'=(v_1,\ldots,v_{i-1},v_i-1,v_{i+1}-1,\ldots,v_d-1)\in\binom{\N}{d}.$$
        We show that $\vb'\le \ub$. Indeed, since $\vb\le\ub+\mathbf{1}$ there exists some $j$ such that $v_{j}< u_j+1$ and $v_k=u_k+1$ for $k\ge j+1$. From $v_{j+1}=u_{j+1}+1>u_j+1\ge j+1$ it follows that $j+1\ge i$. Thus,
        \[
         v_k'=v_k-1=u_k\quad\text{for all } k\ge j+1.
        \]
        If also $j\ge i$, then $v_j'=v_j-1<u_j$, and hence $\vb'\le \ub$. Otherwise, $j=i-1$ and
        \[
         v_k'=v_k=k\le u_k\quad\text{for all } k=1,\dots, j.
        \]
        So again $\vb'\le \ub$. Now we have $\vb=\pi_{v_i-1}(\vb')\in \Inc(\vb')$ and this concludes the proof.	
\end{proof}

\begin{lemma}\label{lm34}
Let $\ub\in \binom{\N}{d}$. Then
$$\Inc(B(\ub))=B(\ub+\mathbf{1}).$$
\end{lemma}

\begin{proof}
The argument is analogous to the one used in the proof of \Cref{lm33}. First, one has $\Inc(B(\ub))\subseteq B(\ub+\mathbf{1})$ because $\vb\le_{\B} \ub+\mathbf{1}$ for all $\vb\in\Inc(\ub)$. To prove the reverse inclusion, we take $\vb\in B(\ub+\mathbf{1})$ and consider the three cases as in the proof of \Cref{lm33}. Since $\vb\le_{\B} \ub+\mathbf{1}$, it is easy to check that in any case the element $\vb'$ defined there satisfies $\vb'\le_{\B} \ub$. Hence, $\vb\in \Inc(\vb')\subseteq  \Inc(B(\ub)).$
\end{proof}

\begin{corollary}
\label{co34}
	If a finite family $\mathcal{F}\subseteq \binom{\N}{d}$ is shifted, then $\Inc(\mathcal{F})$ is also shifted.
\end{corollary}

\begin{proof}
	Since $\mathcal{F}$ is shifted, it has the form
$$
\mathcal{F}=\bigcup_{i=1}^sB(\ub_i)\quad 
\text{ for some }\ \ub_1,\ldots,\ub_s\in \mathcal{F}.
$$
So from \Cref{lm34} it follows that
$$
\Inc(\mathcal{F})=\bigcup_{i=1}^s\Inc(B(\ub_i))=\bigcup_{i=1}^sB(\ub_i+\mathbf{1}).
$$
Hence, $\Inc(\mathcal{F})$ is also shifted.
\end{proof}

\Cref{main} follows from the following result:

\begin{theorem}\label{th33}
Let $\mathcal{F}\subseteq \binom{\N}{d}$ be a finite family of $d$-sets. Then
$$\Inc(\mathcal{C}(\mathcal{F}))\subseteq \mathcal{C}(\Inc(\mathcal{F})).$$
\end{theorem}

By \Cref{lm33}, $\Inc(\mathcal{C}(\mathcal{F}))$ is a compressed set. So the inclusion
\[
 \Inc(\mathcal{C}(\mathcal{F}))\subseteq \mathcal{C}(\Inc(\mathcal{F}))
\]
 is equivalent to
$$
|\Inc(\mathcal{C}(\mathcal{F}))|\le |\mathcal{C}(\Inc(\mathcal{F}))|=|\Inc(\mathcal{F})|.
$$
In particular, we see that Theorems \ref{main} and \ref{th33} are in fact equivalent.

In order to prove \Cref{th33}, we follow the idea of compression due to Clements and Lindstr\"{o}m \cite{CL}. Basically, for a finite family $\mathcal{F}\subseteq \binom{\N}{d}$ we will construct a shifted family $\mathcal{F}^{(\infty)}\subseteq \binom{\N}{d}$ such that
\[
 |\mathcal{F}|=|\mathcal{F}^{(\infty)}|\quad \text{and}\quad|\Inc(\mathcal{F})|\ge|\Inc(\mathcal{F}^{(\infty)})|\ge|\Inc(\mathcal{C}(\mathcal{F}))|.
\]
The construction of the family $\mathcal{F}^{(\infty)}$ is based on the process of taking partial compressions similar to the one used in \cite{CL}: at each step the compression is taken while keeping one coordinate fixed. But unlike in \cite{CL}, we consider here partial compressions that fix either the smallest or the largest element of each set in the family.

Let $\mathcal{F}\subseteq \binom{\N}{d}$ be a finite family. For each $k\geq 1$ set
\begin{align*}
  \mathcal{F}_{1,k}&=\{\ub=(u_1,\ldots,u_d)\in \mathcal{F} \mid u_1=k \}, &\mathcal{F}_{d,k}&=\{\ub=(u_1,\ldots,u_d)\in \mathcal{F} \mid u_d=k \},\\
  \widehat{\mathcal{F}}_{1,k}&=\Big\{\widehat{\ub}\in \binom{\N}{d-1}\mid (k,\widehat{\ub})\in \mathcal{F}_{1,k}\Big\}, &\widehat{\mathcal{F}}_{d,k}&=\Big\{\widehat{\ub}\in \binom{\N}{d-1}\mid (\widehat{\ub},k)\in \mathcal{F}_{d,k}\Big\}.
\end{align*}
Note that the family $\widehat{\mathcal{F}}_{1,k}$ is actually a subset of $\binom{\N_{>k}}{d-1}$. So one can form the compression $\mathcal{C}_{>k}(\widehat{\mathcal{F}}_{1,k})$ of $\widehat{\mathcal{F}}_{1,k}$ in $\binom{\N_{>k}}{d-1}$ (see \Cref{sec2}).
Since $\mathcal{F}$ is finite, note also that
$$
\mathcal{F}_{1,k}=\mathcal{F}_{d,k}=\widehat{\mathcal{F}}_{1,k}=\widehat{\mathcal{F}}_{d,k}=\emptyset\quad
\text{ for }\ k \gg 0.
$$

\begin{definition}
\label{compression}
 Using the above notation, we define
 \begin{enumerate}
  \item
  the {\em left partial compression} of $\mathcal{F}_{1,k}$ as
  $$\mathcal{C}^{(l)}(\mathcal{F}_{1,k})=\{(k,\widehat{\ub})\mid \widehat{\ub}\in \mathcal{C}_{>k}(\widehat{\mathcal{F}}_{1,k})\},$$
  and, more generally, the {\em left partial compression} of $\mathcal{F}$ as
  \begin{equation}
  \label{p1}
   \mathcal{C}^{(l)}(\mathcal{F})=\bigcup_{k\geq 1}\mathcal{C}^{(l)}(\mathcal{F}_{1,k}),
  \end{equation}

  \item
  the {\em right partial compression} of $\mathcal{F}_{d,k}$ as
  $$\mathcal{C}^{(r)}(\mathcal{F}_{d,k})=\{(\widehat{\ub},k)\mid \widehat{\ub}\in \mathcal{C}(\widehat{\mathcal{F}}_{d,k})\},$$
  and, more generally, the {\em right partial compression} of $\mathcal{F}$ as
  \begin{equation}
  \label{pd}
   \mathcal{C}^{(r)}(\mathcal{F})=\bigcup_{k \geq 1}\mathcal{C}^{(r)}(\mathcal{F}_{d,k}).
  \end{equation}
 \end{enumerate}
 The family $\mathcal{F}$ is called \emph{left-compressed} (respectively, \emph{right-compressed}) if
\[
\mathcal{F}=\mathcal{C}^{(l)}(\mathcal{F})\quad (\text{respectively, }\ \mathcal{F}=\mathcal{C}^{(r)}(\mathcal{F})).                                                                                                                                                                                                                                   \]
In other words, $\mathcal{F}$ is \emph{left-compressed} (respectively, \emph{right-compressed}) if and only if the family $\widehat{\mathcal{F}}_{1,k}$ is compressed in $\binom{\N_{>k}}{d-1}$ (respectively, $\widehat{\mathcal{F}}_{d,k}$ is compressed in $\binom{\N}{d-1}$) for every $k\geq 1$.
\end{definition}

\begin{example}
\label{ex1}
 Let $\mathcal{F}=\{(1,2,6),\;(1,3,5),\;(2,3,5),\;(3,5,6)\}\subseteq \binom{\N}{3}$. Then
 \begin{align*}
  \mathcal{F}_{1,1}&=\{(1,2,6),\;(1,3,5)\},\ \mathcal{F}_{1,2}=\{(2,3,5)\},\ \mathcal{F}_{1,3}=\{(3,5,6)\},\\
  \mathcal{F}_{3,5}&=\{(1,3,5),\;(2,3,5)\},\ \mathcal{F}_{3,6}=\{(1,2,6),\;(3,5,6)\}.
 \end{align*}
 By definition,
 \begin{align*}
  \mathcal{C}^{(l)}(\mathcal{F}_{1,1})&=\{(1,2,3),\;(1,2,4)\},\ \mathcal{C}^{(l)}(\mathcal{F}_{1,2})=\{(2,3,4)\},\ \mathcal{C}^{(l)}(\mathcal{F}_{1,3})=\{(3,4,5)\},\\
  \mathcal{C}^{(r)}(\mathcal{F}_{3,5})&=\{(1,2,5),\;(1,3,5)\},\ \mathcal{C}^{(r)}(\mathcal{F}_{3,6})=\{(1,2,6),\;(1,3,6)\}.
 \end{align*}
 Hence,
 \begin{align*}
  \mathcal{C}^{(l)}(\mathcal{F})&=\{(1,2,3),\;(1,2,4),\; (2,3,4),\; (3,4,5)\},\\
  \mathcal{C}^{(r)}(\mathcal{F})&=\{(1,2,5),\;(1,3,5),\; (1,2,6),\;(1,3,6)\}.
 \end{align*}

\end{example}

Some elementary properties of partial compressions are listed in the next lemma. To state the result we need to extend the squashed order to finite subsets of $\binom{\N}{d}$: for such two subsets $\mathcal{F},\mathcal{F}'$
we write
$$\mathcal{F}\le\mathcal{F}'$$
if the largest element (w.r.t.~the squashed order)
of $(\mathcal{F}\backslash\mathcal{F}')\cup (\mathcal{F}'\backslash\mathcal{F})$
belongs to $\mathcal{F}'$. Thus, a finite compressed set in $\binom{\N}{d}$ is minimal among all the subsets of $\binom{\N}{d}$ of its size.

\begin{lemma}
\label{properties}
 Let $\mathcal{F}\subseteq \binom{\N}{d}$ be a finite subset. Then:
 \begin{enumerate}
  \item
  For every $k\geq 1$, one has
  \[
   (\mathcal{C}^{(l)}(\mathcal{F}))_{1,k}=\mathcal{C}^{(l)}(\mathcal{F}_{1,k})\quad\text{and}\quad
   (\mathcal{C}^{(r)}(\mathcal{F}))_{d,k}=\mathcal{C}^{(r)}(\mathcal{F}_{d,k}),
  \]
  or in other words,
  \[
   \widehat{\mathcal{C}^{(l)}(\mathcal{F})}_{1,k}=\mathcal{C}_{>k}(\widehat{\mathcal{F}}_{1,k})\quad\text{and}\quad
   \widehat{\mathcal{C}^{(r)}(\mathcal{F})}_{d,k}=\mathcal{C}(\widehat{\mathcal{F}}_{d,k}).
  \]
  \item
  $|\mathcal{F}|=|\mathcal{C}^{(l)}(\mathcal{F})|=|\mathcal{C}^{(r)}(\mathcal{F})|$.

  \item
  $\mathcal{C}^{(l)}(\mathcal{F})\le \mathcal{F}$ and $\mathcal{C}^{(r)}(\mathcal{F})\le \mathcal{F}$.
 \end{enumerate}
\end{lemma}

\begin{proof}
 (i) is immediate from \Cref{compression}. From (i) we get (ii) since the unions in \eqref{p1} and \eqref{pd} as well as in $\mathcal{F}=\bigcup_{k\geq 1}\mathcal{F}_{1,k}=\bigcup_{k\geq 1}\mathcal{F}_{d,k}$ are disjoint. For (iii), it suffices to note that
 \[
  \mathcal{C}^{(l)}(\mathcal{F}_{1,k})\le \mathcal{F}_{1,k}\quad\text{and}\quad \mathcal{C}^{(r)}(\mathcal{F}_{d,k})\le \mathcal{F}_{d,k}\quad \text{for every } k\geq 1,
 \]
which follow easily from the obvious relations: $\mathcal{C}_{>k}(\widehat{\mathcal{F}}_{1,k})\le \widehat{\mathcal{F}}_{1,k}$ and $\mathcal{C}(\widehat{\mathcal{F}}_{d,k})\le \widehat{\mathcal{F}}_{d,k}$.
\end{proof}

We now use partial compressions to construct a shifted family from a given family of sets. For a finite family $\mathcal{F}\subseteq \binom{\N}{d}$ consider the following sequence of families in $\binom{\N}{d}$ which is obtained by alternatively applying the left and right partial compressions:
\begin{align*}
	&\mathcal{F}^{(0)}=\mathcal{F},\ \ \mathcal{F}^{(1)}=\mathcal{C}^{(l)}(\mathcal{F}^{(0)}),\ \ \mathcal{F}^{(2)}=\mathcal{C}^{(r)}(\mathcal{F}^{(1)}),\cdots\\
	&\mathcal{F}^{(2k-1)}=\mathcal{C}^{(l)}(\mathcal{F}^{(2k-2)}),\ \ \mathcal{F}^{(2k)}=\mathcal{C}^{(r)}(\mathcal{F}^{(2k-1)}),\cdots.
\end{align*}
According to \Cref{properties}(iii),
$$\mathcal{F}\ge \mathcal{F}^{(1)}\ge \mathcal{F}^{(2)}\ge\cdots\ge \mathcal{F}^{(2k-1)}\ge \mathcal{F}^{(2k)}\ge\cdots.$$
Since the squashed order is a well-order on $\binom{\N}{d}$, the previous sequence must stabilize, i.e. there exists some $j$ such that $\mathcal{F}^{(k)}=\mathcal{F}^{(j)}$ for all $k\ge j$. Let $\mathcal{F}^{(\infty)}$ denote this family of sets in the limit. Then
$$\mathcal{C}^{(l)}(\mathcal{F}^{(\infty)})=\mathcal{F}^{(\infty)}\quad \text{and}\quad \mathcal{C}^{(r)}(\mathcal{F}^{(\infty)})=\mathcal{F}^{(\infty)}.$$
Hence, $\mathcal{F}^{(\infty)}$ is left- and right-compressed. The next result shows that $\mathcal{F}^{(\infty)}$ is a shifted family.

\begin{lemma}\label{lm38}
	Let $d\ge 2$ and let $\mathcal{F}\subseteq \binom{\N}{d}$ be a finite family of $d$-sets. If $\mathcal{F}$ is left- and right-compressed, then $\mathcal{F}$ is shifted. In particular, $\mathcal{F}^{(\infty)}$ is always shifted.
\end{lemma}

\begin{proof}
 Let $\ub=(u_1,\dots,u_d)\in\mathcal{F}$ and $\vb=(v_1,\dots,v_d)\in\binom{\N}{d}$ with $\vb\le_{\B}\ub$. We have to show that $\vb\in\mathcal{F}$. Consider $\wb=(v_1,\dots,v_{d-1},u_d)$. It is clear that $\vb\le_{\B}\wb\le_{\B}\ub$. Since $\mathcal{F}$ is right-compressed, $\widehat{\mathcal{F}}_{d,u_d}$ is compressed. So from $\widehat{\ub}=(u_1,\dots,u_{d-1})\in \widehat{\mathcal{F}}_{d,u_d}$ and
 \[
  \widehat{\wb}=(v_1,\dots,v_{d-1})\le\widehat{\ub}
 \]
it follows that $\widehat{\wb}\in\widehat{\mathcal{F}}_{d,u_d}$. Hence, $\wb\in\mathcal{F}$. Note that $\vb$ and $\wb$ have the same smallest element, namely $v_1$. So using the assumption that $\mathcal{F}$ is left-compressed and an analogous argument as above we conclude that $\vb\in\mathcal{F}$.
\end{proof}

\begin{example}
 Consider again the family $\mathcal{F}\subseteq \binom{\N}{3}$ in \Cref{ex1}. We know that
 \[
  \mathcal{F}^{(1)}=\mathcal{C}^{(l)}(\mathcal{F})=\{(1,2,3),\;(1,2,4),\; (2,3,4),\; (3,4,5)\}.
 \]
Hence,
\[
 \mathcal{F}^{(2)}=\mathcal{C}^{(r)}(\mathcal{F}^{(1)})=\{(1,2,3),\;(1,2,4),\; (1,3,4),\; (1,2,5)\}.
\]
Clearly, $\mathcal{F}^{(2)}$ is left- and right-compressed. Thus, $\mathcal{F}^{(\infty)}=\mathcal{F}^{(2)}$.
\end{example}

The next lemma examines the compositions of $\Inc$ and partial compressions. Recall that the map $\pi_1\in\Inc$ is defined by
 $\pi_1(j)=j+1$ for all  $j\in \N$.

\begin{lemma}
\label{compositions}
Let $\mathcal{F}\subseteq \binom{\N}{d}$ be a finite family of $d$-sets. Then
 \begin{align}
  \mathcal{C}^{(l)}(\Inc(\mathcal{F}))&=\bigcup_{k\geq 1}\{(k,\widehat{\ub})\mid \widehat{\ub}\in \mathcal{C}_{>k}\big(\Inc(\widehat{\mathcal{F}}_{1,k})\cup \pi_1(\widehat{\mathcal{F}}_{1,k-1})\big)\},\label{eq1}\\
  \mathcal{C}^{(r)}(\Inc(\mathcal{F}))&=\bigcup_{k\geq 1}\{(\widehat{\ub},k)\mid \widehat{\ub}\in \mathcal{C}\big(\widehat{\mathcal{F}}_{d,k}\cup\Inc(\widehat{\mathcal{F}}_{d,k-1})\big)\},\label{eq2}\\
  \Inc(\mathcal{C}^{(l)}(\mathcal{F}))&=\bigcup_{k\geq 1}\{(k,\widehat{\ub})\mid \widehat{\ub}\in \Inc\big(\mathcal{C}_{>k}(\widehat{\mathcal{F}}_{1,k})\big)\cup \pi_1\big(\mathcal{C}_{>k-1}(\widehat{\mathcal{F}}_{1,k-1})\big)\},\label{eq3}\\
\Inc(\mathcal{C}^{(r)}(\mathcal{F}))&=\bigcup_{k\geq 1}\{(\widehat{\ub},k)\mid \widehat{\ub}\in \mathcal{C}(\widehat{\mathcal{F}}_{d,k})\cup\Inc\big(\mathcal{C}(\widehat{\mathcal{F}}_{d,k-1})\big)\}.\label{eq4}
 \end{align}
\end{lemma}

\begin{proof}
 For any $\ub=(u_1,\widehat{\ub})=(\widehat{\ub}',u_d)\in\binom{\N}{d}$ it follows from \eqref{Inc} that
 \begin{align*}
  \Inc(\ub)&=\{(u_1,\widehat{\vb})\mid \widehat{\vb}\in\Inc(\widehat{\ub})\}\cup\{(u_1+1,\pi_1(\widehat{\ub}))\}\\
  &=\{(\widehat{\ub}',u_d)\}\cup\{(\widehat{\vb},u_d+1)\mid \widehat{\vb}\in\Inc(\widehat{\ub}')\}.
 \end{align*}
This yields
\begin{align*}
 \Inc(\mathcal{F}_{1,k})&=\bigcup_{{\ub}\in{\mathcal{F}}_{1,k}}\Inc({\ub})=\bigcup_{\widehat{\ub}\in\widehat{\mathcal{F}}_{1,k}}\Inc\big((k,\widehat{\ub})\big)\\
 &=\{(k,\widehat{\vb})\mid \widehat{\vb}\in \Inc(\widehat{\mathcal{F}}_{1,k})\}\cup\{(k+1,\widehat{\vb})\mid \widehat{\vb}\in \pi_1(\widehat{\mathcal{F}}_{1,k})\},\\
 \Inc(\mathcal{F}_{d,k})&=\bigcup_{{\ub}\in{\mathcal{F}}_{d,k}}\Inc({\ub})=\bigcup_{\widehat{\ub}\in\widehat{\mathcal{F}}_{d,k}}\Inc\big((\widehat{\ub},k)\big)\\
 &=\{(\widehat{\ub},k)\mid \widehat{\ub}\in \widehat{\mathcal{F}}_{d,k}\}\cup\{(\widehat{\vb},k+1)\mid \widehat{\vb}\in \Inc(\widehat{\mathcal{F}}_{d,k})\}.
\end{align*}
Hence, we get the following equations which imply \eqref{eq1} and \eqref{eq2}:
 \begin{align}
 \Inc(\mathcal{F})&=\bigcup_{k\geq 1}\Inc(\mathcal{F}_{1,k})=\bigcup_{k\geq 1}\{(k,\widehat{\vb})\mid \widehat{\vb}\in \Inc(\widehat{\mathcal{F}}_{1,k})\cup \pi_1(\widehat{\mathcal{F}}_{1,k-1})\},\label{eq6}\\
 \Inc(\mathcal{F})&=\bigcup_{k\geq 1}\Inc(\mathcal{F}_{d,k})=\bigcup_{k\geq 1}\{(\widehat{\vb},k)\mid \widehat{\vb}\in \widehat{\mathcal{F}}_{d,k}\cup \Inc(\widehat{\mathcal{F}}_{d,k-1})\}\label{eq7}.
 \end{align}

Replacing $\mathcal{F}$ by $\mathcal{C}^{(l)}(\mathcal{F})$ in \eqref{eq6} and by $\mathcal{C}^{(r)}(\mathcal{F})$ in \eqref{eq7} and then using \Cref{properties}(i) we obtain \eqref{eq3} and \eqref{eq4}.
\end{proof}

We will prove \Cref{th33} by induction on $d$. The next lemma allows us to apply the induction hypothesis.
We say that \Cref{th33} is {\em true for $d$} if for all finite families $\mathcal{F}\subseteq \binom{\N}{d}$ it holds that $\Inc(\mathcal{C}(\mathcal{F}))\subseteq \mathcal{C}(\Inc(\mathcal{F})).$

\begin{lemma}\label{lm37}
Let $d\ge 2$ and assume that \Cref{th33} is true for $d-1$. Then the following statements hold:
\begin{enumerate}
 \item
 For all $k\in\N$ and all finite families $\mathcal{G}\subseteq\binom{\N_{>k}}{d-1}$ one has
 \[
  \Inc(\mathcal{C}_{>k}(\mathcal{G}))\subseteq \mathcal{C}_{>k}(\Inc(\mathcal{G})).
 \]
 \item
 For all finite families $\mathcal{F}\subseteq \binom{\N}{d}$ one has
\begin{align*}
\Inc(\mathcal{C}^{(l)}(\mathcal{F}))&\subseteq \mathcal{C}^{(l)}(\Inc(\mathcal{F})),\\
\Inc(\mathcal{C}^{(r)}(\mathcal{F}))&\subseteq \mathcal{C}^{(r)}(\Inc(\mathcal{F})).
\end{align*}
\end{enumerate}
\end{lemma}

\begin{proof}
(i) The map $\pi_1^k\colon\binom{\N}{d-1}\to \binom{\N_{>k}}{d-1}$ given by
\[
 \pi_1^k(\ub)=\ub+\mathbf{k}:=(u_1+k,\dots,u_d+k)\quad\text{for every } \ub=(u_1,\dots,u_d)\in\binom{\N}{d-1}
\]
is a bijection. Let $\sigma$ denote the inverse of $\pi_1^k$. Then for finite families $\mathcal{G}\subseteq\binom{\N_{>k}}{d-1}$ and $\mathcal{G}'\subseteq\binom{\N}{d-1}$ it is apparent that
\[
 \mathcal{C}_{>k}(\mathcal{G})=\pi_1^k(\mathcal{C}(\sigma(\mathcal{G}))),\ \ \Inc(\sigma(\mathcal{G}))=\sigma(\Inc(\mathcal{G})),\ \ \Inc(\pi_1^k(\mathcal{G}'))=\pi_1^k(\Inc(\mathcal{G}')).
\]
So the assumption gives
\[
 \begin{aligned}
  \Inc(\mathcal{C}_{>k}(\mathcal{G}))&=\Inc(\pi_1^k(\mathcal{C}(\sigma(\mathcal{G}))))=\pi_1^k(\Inc(\mathcal{C}(\sigma(\mathcal{G}))))\\
  &\subseteq \pi_1^k(\mathcal{C}(\Inc(\sigma(\mathcal{G})))) = \pi_1^k(\mathcal{C}(\sigma(\Inc(\mathcal{G}))))=\mathcal{C}_{>k}(\Inc(\mathcal{G})).
 \end{aligned}
\]

(ii) It is easily seen that for all $k\in\N$ and all finite families $\mathcal{G},\mathcal{G}'\subseteq\binom{\N_{>k}}{d-1}$ one has
\[
 \mathcal{C}_{>k}(\mathcal{G}\cup\mathcal{G}')\supseteq\mathcal{C}_{>k}(\mathcal{G})\cup\mathcal{C}_{>k}(\mathcal{G}')\quad \text{and}\quad
 \mathcal{C}_{>k+1}(\pi_1(\mathcal{G}))=\pi_1(\mathcal{C}_{>k}(\mathcal{G})).
\]
So from \eqref{eq1}, \eqref{eq3}, and (i) it follows that
\begin{align*}
\mathcal{C}^{(l)}(\Inc(\mathcal{F}))&=\bigcup_{k\geq 1}\{(k,\widehat{\ub})\mid \widehat{\ub}\in \mathcal{C}_{>k}\big(\Inc(\widehat{\mathcal{F}}_{1,k})\cup \pi_1(\widehat{\mathcal{F}}_{1,k-1})\big)\}\\
&\supseteq\bigcup_{k\geq 1}\{(k,\widehat{\ub})\mid \widehat{\ub}\in \mathcal{C}_{>k}\big(\Inc(\widehat{\mathcal{F}}_{1,k})\big)\cup \mathcal{C}_{>k}\big(\pi_1(\widehat{\mathcal{F}}_{1,k-1})\big)\}\\
&\supseteq\bigcup_{k\geq 1}\{(k,\widehat{\ub})\mid \widehat{\ub}\in \Inc\big(\mathcal{C}_{>k}(\widehat{\mathcal{F}}_{1,k})\big)\cup \pi_1\big(\mathcal{C}_{>k-1}(\widehat{\mathcal{F}}_{1,k-1})\big)\}\\
&=\Inc(\mathcal{C}^{(l)}(\mathcal{F})).
\end{align*}

The inclusion $\Inc(\mathcal{C}^{(r)}(\mathcal{F}))\subseteq \mathcal{C}^{(r)}(\Inc(\mathcal{F}))$ is proved similarly.
\end{proof}

We are now ready to prove \Cref{th33}.

\begin{proof}[Proof of  \Cref{th33}]
We argue by induction on $d$. The case $d=1$ is easy to check. Assume that $d\ge 2$ and that the theorem is true for $d-1$. Let $\mathcal{F}\subseteq \binom{\N}{d}$ be a finite family. By \Cref{lm33}, it suffices to prove that
$$|\Inc(\mathcal{C}(\mathcal{F}))|\le |\mathcal{C}(\Inc(\mathcal{F}))|=|\Inc(\mathcal{F})|.$$
From \Cref{properties}(ii) and \Cref{lm37}(ii) it follows that
\begin{align}	\label{eq5}
\begin{split}
|\Inc(\mathcal{F})|&=|\mathcal{C}^{(l)}(\Inc(\mathcal{F}))|\\
&\ge |\Inc(\mathcal{C}^{(l)}(\mathcal{F}))|=|\Inc(\mathcal{F}^{(1)})|
=|\mathcal{C}^{(r)}(\Inc(\mathcal{F}^{(1)}))|\\
&\ge |\Inc(\mathcal{C}^{(r)}(\mathcal{F}^{(1)}))|=|\Inc(\mathcal{F}^{(2)})|=|\mathcal{C}^{(l)}(\Inc(\mathcal{F}^{(2)}))|\\
&\ge \cdots\ge |\Inc(\mathcal{F}^{(\infty)})|.
\end{split}
\end{align}
Thus, to complete the proof it is enough to show that
$$|\Inc(\mathcal{C}(\mathcal{F}))|\le |\Inc(\mathcal{F}^{(\infty)})|.$$
This follows from the following

\begin{claim}
 Let $\mathcal{G},\mathcal{H}\subseteq \binom{\N}{d}$ be finite families of $d$-sets with $|\mathcal{G}|=|\mathcal{H}|$. If $\mathcal{G}$ is compressed and $\mathcal{H}$ is left- and right-compressed, then $|\Inc(\mathcal{G})|\le |\Inc(\mathcal{H})|.$
\end{claim}

Suppose the claim is false and let $\mathcal{H}$ be the minimal left- and right-compressed family which violates the claim. Then $\mathcal{H}> \mathcal{G}$ because $\mathcal{G}$ is compressed. Let $\ub=\max \mathcal{H}$ and $\vb=\min(\mathcal{G}\setminus \mathcal{H}).$ Evidently, $\ub>\vb.$ Consider the family
$$\mathcal{K}=(\mathcal{H}\setminus\{\ub\})\cup\{\vb\}.$$
Then $\mathcal{K}<\mathcal{H}$. We will show that
\begin{equation}
\label{eq13}
 |\Inc(\mathcal{K})|\le |\Inc(\mathcal{H})|.
\end{equation}
If this is true, then combining it with \eqref{eq5} one gets
$$|\Inc(\mathcal{K}^{(\infty)})|\le |\Inc(\mathcal{K})|\le |\Inc(\mathcal{H})|.$$
This implies that the family $\mathcal{K}^{(\infty)}$, which is left- and right-compressed by virtue of \Cref{lm38}, also violates the above claim. But this contradicts the minimality of $\mathcal{H}$, because $\mathcal{K}^{(\infty)}\le \mathcal{K}<\mathcal{H}$ by \Cref{properties}(iii).

So it remains to prove \eqref{eq13}.
We first show that $u_1<v_1$. Indeed, if $v_1\le u_1$, then
$$\ub'=(v_1,u_2,\ldots,u_d)\le_{\B} (u_1,u_2,\ldots,u_d)=\ub.$$
This implies $\ub'\in \mathcal{H}$ since $\mathcal{H}$ is shifted by \Cref{lm38}. Thus,
$$(u_2,\ldots,u_{d})\in \widehat{\mathcal{H}}_{1,v_1}.$$
Note that $\widehat{\mathcal{H}}_{1,v_1}$ is compressed, because $\mathcal{H}$ is left-compressed. So from
$$(v_2,\ldots,v_{d})\le (u_2,\ldots,u_{d})$$
(this is true as $\vb<\ub$) it follows that $(v_2,\ldots,v_{d})\in \widehat{\mathcal{H}}_{1,v_1}.$ This yields $\vb\in \mathcal{H}$, which is a contradiction. Thus, we must have $u_1<v_1$.

Next, we show that
\begin{equation}
\label{eq14}
 \Inc(\vb)\setminus\{\vb+\mathbf{1}\}\subseteq \Inc(\mathcal{H}\setminus\{\ub\}).
\end{equation}
Let $\wb\in\Inc(\vb)\setminus\{\vb+\mathbf{1}\}$. Then there exists some $i\ge 1$ such that
\[
 \wb=(v_1,\dots,v_i,v_{i+1}+1,\dots,v_d+1).
\]
Set
\[
 \wb'=\wb-\mathbf{1}=(v_1-1,\dots,v_i-1,v_{i+1},\dots,v_d).
\]
Note that $\wb'\in\binom{\N}{d}$ since $v_1>u_1\ge 1$. Obviously, $\wb'<\vb$. This implies $\wb'\in \mathcal{G}$, and hence $\wb'\in \mathcal{H}$ as $\vb=\min(\mathcal{G}\setminus \mathcal{H}).$ Moreover, $\wb'\in \mathcal{H}\setminus\{\ub\}$ since $\wb'<\vb<\ub$.
It follows that
$$\wb=\pi_1(\wb')\in\Inc(\wb')\subseteq \Inc(\mathcal{H}\setminus\{\ub\}),$$
which verifies \eqref{eq14}.

Now for the family $\mathcal{K}=(\mathcal{H}\setminus\{\ub\})\cup\{\vb\}$ one has
$$\Inc(\mathcal{K})=\Inc(\mathcal{H}\setminus\{\ub\})\cup\Inc(\vb)=\Inc(\mathcal{H}\setminus\{\ub\})\cup\{\vb+\mathbf{1}\}.$$
As $\ub=\max \mathcal{H}$ it is clear that $\ub+\mathbf{1}\in \Inc(\mathcal{H})\setminus \Inc(\mathcal{H}\setminus\{\ub\})$. Thus,
\[
 |\Inc(\mathcal{K})|\le|\Inc(\mathcal{H}\setminus\{\ub\})|+1\le |\Inc(\mathcal{H})|.
\]
This proves \eqref{eq13}, and hence concludes the proof of the theorem.
\end{proof}

Next we derive a numerical version of \Cref{th33}. For a positive integer $m$ with the $d$-binomial representation \eqref{binom}
we set
\[
  \Inc^{[d]}(m)= \binom{a_d+1}{d}+\binom{a_{d-1}+1}{d-1}+\cdots+\binom{a_s+1}{s}.
\]
Also let $\Inc^{[d]}(0)=0$. This notation is motivated by the fact that if $|C(\ub)|=m$ for some $\ub\in\binom{\N}{d}$, then $|C(\ub+\mathbf{1})|=\Inc^{[d]}(m)$. Indeed, $|C(\ub)|=m$ means that $\ub$ is the $m$-th element of $\binom{\N}{d}$ in the squashed order.
So $m$ has the $d$-binomial representation \eqref{binom} if and only if
\[
 {\ub}=(a_s-s+1,\dots,a_s-1,a_s,a_{s+1}+1,\dots,a_{d-1}+1,a_{d}+1)
\]
(see \cite[p. 117]{An}). This gives
\[
 \ub+\mathbf{1}=(a_s-s+2,\dots,a_s,a_s+1,a_{s+1}+2,\dots,a_{d-1}+2,a_{d}+2),
\]
and thus
\[
 |C(\ub+\mathbf{1})|=\binom{a_d+1}{d}+\binom{a_{d-1}+1}{d-1}+\cdots+\binom{a_s+1}{s}=\Inc^{[d]}(m).
\]

\begin{corollary}
\label{numerical}
 For any finite family $\mathcal{F}\subseteq \binom{\N}{d}$ one has
 \[
  |\Inc(\mathcal{F})|\ge\Inc^{[d]}(|\mathcal{F}|).
 \]
\end{corollary}

\begin{proof}
 Since $\mathcal{C}(\mathcal{F})$ is compressed, it holds that $\mathcal{C}(\mathcal{F})=C(\ub)$ with $\ub=\max\mathcal{C}(\mathcal{F})$. By \Cref{lm33}, $\Inc(\mathcal{C}(\mathcal{F}))=C(\ub+\mathbf{1})$. So as explained above, this gives
 \[
  |\Inc(\mathcal{C}(\mathcal{F}))|=|C(\ub+\mathbf{1})|=\Inc^{[d]}(|C(\ub)|)=\Inc^{[d]}(|\mathcal{F}|).
\]
The desired conclusion now follows from \Cref{th33}.
\end{proof}

\begin{remark}
\label{re}
 The combinatorial shifting technique, which is usually used to prove the Kruskal-Katona theorem and has many other applications in extremal set theory (see, e.g., \cite{Fr}), cannot be applied to prove \Cref{th33}. For $\mathcal{F}\subseteq \binom{\N}{d}$ and $i>1$ set
 \begin{align*}
  S_i(\ub)&=
  \begin{cases}
   (\ub\setminus \{i\})\cup\{1\}&\text{if }\ i\in\ub,\; 1\not\in\ub,\;(\ub\setminus \{i\})\cup\{1\}\not\in \mathcal{F},\\
   \ub&\text{otherwise},
  \end{cases}\\
  S_i(\mathcal{F})&=\{S_i(\ub)\mid \ub\in\mathcal{F}\}.
 \end{align*}
 Then one has the following inclusion
 \[
  \partial(S_i(\mathcal{F}))\subseteq S_i(\partial(\mathcal{F})),
 \]
 which is essential for the proof of the Kruskal-Katona theorem using the combinatorial shifting technique. However, as one can easily check, there is no inclusion relation between $\Inc(S_i(\mathcal{F}))$ and $S_i(\Inc(\mathcal{F}))$ in general.
\end{remark}

%%%%%%%%%%%%%%%%%%%%%%%%%%%%%%%%%%%
\section{Simplicial complexes and Inc action}
\label{sec4}

In this section we derive some consequences when applying the results of the previous section to simplicial complexes. The main outcome is \Cref{thm:onemainresult} where we characterize all possible $f$-vectors of chains of simplicial complexes that are combinatorial invariant under the action of the monoid $\Inc$.

Recall that a \emph{simplicial complex} $\Delta$ on the vertex set $[m]$ is a collection of subsets of $[m]$ that is closed under inclusion, i.e., if $F\in \Delta$ and $G\subseteq F$, then $G\in\Delta$. In other words, $\Delta$ is a simplicial complex if and only if $\partial\Delta\subseteq\Delta$.

Let $\Delta$ be a simplicial complex. Elements of $\Delta$ are called \emph{faces}. For $d\ge 1$ let $F_d(\Delta)$ be the subset of $\Delta$ consisting of all faces $F$ with $|F|=d$. We say that $\Delta$ is a \emph{shifted} (respectively, \emph{compressed}) \emph{complex} if $F_d(\Delta)$ is a shifted (respectively, {compressed}) family for every $d\ge 1$. Set $f_{d-1}(\Delta)=|F_d(\Delta)|$. Then the vector
\[
 {\bf f}(\Delta)=(f_{0}(\Delta),f_1(\Delta),f_2(\Delta),\dots)
\]
is called the \emph{$f$-vector} of $\Delta$. Note that $f_{d}(\Delta)=0$ for $d\gg 0$. Hence ${\bf f}(\Delta)$ belongs to the set
\[
 \Z_{\ge0}^{(\infty)} :=\{(n_i)_{i\ge0}\mid n_i\in\Z_{\ge0}\ \text{ and }\ n_i=0\ \text{ for }\ i\gg0\}.
\]
The numerical version of the Kruskal-Katona theorem provided in \Cref{sec2} leads to the following characterization of $f$-vectors of simplicial complexes (see \cite[Theorem 8.5]{GK}): a sequence ${\bf f}=(f_{0},f_1,\dots)\in\Z_{\ge0}^{(\infty)}$ is the $f$-vector of a simplicial complex if and only if
\[
 \partial_d(f_{d})\le f_{d-1}\quad \text{for all}\ d\ge 1.
\]

Recall that the {$\Inc$-image} of $\Delta$ is defined by
\[
 \Inc(\Delta)=\bigcup_{d\ge 1}\Inc(F_d(\Delta)).
\]
As one can easily check, $\Inc(\Delta)$ is a simplicial complex.
The next result is an immediate consequence of \Cref{lm33} and \Cref{co34}.

\begin{corollary}
 If $\Delta$ is a shifted (respectively, {compressed}) complex, then so is $\Inc(\Delta)$.
\end{corollary}

The compression $\mathcal{C}(\Delta)$ of $\Delta$ is defined in an obvious way. By \Cref{th33} we get

\begin{corollary}
 For any simplicial complex $\Delta$ one has
 \[
  \Inc(\mathcal{C}(\Delta))\subseteq\mathcal{C}(\Inc(\Delta)).
 \]
\end{corollary}

Next, from \Cref{numerical} we obtain the following relation between the $f$-vectors of $\Delta$ and $\Inc(\Delta)$.

\begin{corollary}
\label{co43}
 For any simplicial complex $\Delta$ one has
 \[
  f_d(\Inc(\Delta))\ge \Inc^{[d+1]}(f_d(\Delta)) \quad\text{for }\ d\ge0.
 \]
\end{corollary}

This result yields a characterization of $f$-vectors of chains of simplicial complexes that are combinatorial invariant under the action of the monoid $\Inc$. For each $n\ge1$ let $\Delta_n$ be a simplicial complex. We say that the chain $(\Delta_n)_{n\ge 1}$ is \emph{\cinv} if
\[
 \Inc(\Delta_n)\subseteq\Delta_{n+1}\quad\text{for all }\ n\ge 1.
\]
Notice that this notion is totally different from the notion of $\Inc$-invariant chains of simplicial complexes given in the introduction, which stems from an algebraic notion of $\Inc$-invariant chains of ideals.

Assume that the chain of simplicial complexes $(\Delta_n)_{n\ge 1}$ is \cinv. Then the chain of $f$-vectors $({\bf f}(\Delta_n))_{n\ge1}$ must satisfy the inequalities posed by the Kruskal-Katona theorem and \Cref{co43}. The next result shows that these inequalities are enough to characterize such chains of $f$-vectors.

\begin{theorem}
\label{thm:onemainresult}
 For each $n\ge 1$ let ${\bf f}_n=(f_{n,0},f_{n,1},\dots)\in\Z_{\ge0}^{(\infty)}$. The following conditions are equivalent:
 \begin{enumerate}
  \item
  There exists a {\cinv} chain of simplicial complexes $(\Delta_n)_{n\ge 1}$ such that ${\bf f}_n={\bf f}(\Delta_n)$ for all $n\ge 1$.
  \item
  For all $n,d\ge 1$ one has
  \[
   \partial_d(f_{n,d})\le f_{n,d-1}\quad\text{and}\quad f_{n+1,d-1}\ge \Inc^{[d]}(f_{n,d-1}).
  \]
 \end{enumerate}
\end{theorem}

\begin{proof}
 The implication (i)$\Rightarrow$(ii) follows immediately from the Kruskal-Katona theorem and \Cref{co43}. Let us prove (ii)$\Rightarrow$(i). For $n,d\ge 1$ let $F_d(\Delta_n)$ be the compressed subset of $\binom{\N}{d}$ of cardinality $f_{n,d-1}$. Now set
\[
 \Delta_n=\bigcup_{d\ge 1}F_d(\Delta)\quad \text{for }\ n\ge 1.
\]
 Then the inequalities $\partial_d(f_{n,d})\le f_{n,d-1}$ imply that $\Delta_n$ is a simplicial complex, whereas the inequalities $f_{n+1,d-1}\ge \Inc^{[d]}(f_{n,d-1})$ imply that $\Inc(\Delta_n)\subseteq\Delta_{n+1}$. In other words, $(\Delta_n)_{n\ge 1}$ is a {\cinv} chain of simplicial complexes. By construction, it is evident that ${\bf f}_n={\bf f}(\Delta_n)$ for all $n\ge 1$.
\end{proof}

%%%%%%%%%%%%%%%%%%%%%%%%%%%%%%%%%%%%%%%%%%%%%%%%%%%%
\section{Open Problems}\label{sec5}

Here we give some problems which might be of interest. The first one arises naturally from \Cref{pb} and the Clements-Lindstr\"{o}m theorem \cite{CL}.

\begin{problem}
 Study \Cref{pb} for multisets.
\end{problem}

The next problem is motivated by \Cref{th33}. 
\begin{problem}
Characterize finite families $\mathcal{F}\subseteq \binom{\N}{d}$ for which the inclusion in \Cref{th33} (or equivalently the inequality in \Cref{numerical}) becomes an equality.
\end{problem}

\begin{remark}
Observe that in \Cref{th33} %(or equivalently  the inequality in \Cref{numerical}) 
the inclusion  becomes an equality for compressed families of sets. But these are not the only examples. For instance, one can take $\mathcal{F}=\{\pi_i(\ub),\pi_{i+1}(\ub),\dots,\pi_j(\ub)\}$ for some $j\ge i\ge 1$ and some $\ub\in \binom{\N}{d}$ (see \Cref{sec2} for the definition of $\pi_i$).
\end{remark}

Given \Cref{thm:onemainresult} it is of interest to consider the following
\begin{problem}
Study  properties of $h$-vectors of {\cinv} chains of simplicial complexes $(\Delta_n)_{n\ge 1}$.
\end{problem}

The last problem concerns stability of chains of simplicial complexes. Let $(\Delta_n)_{n\ge 1}$ be a chain of simplicial complexes with $\Delta_n$ having the vertex set $[n]$. As mentioned in the introduction, if the chain $(\Delta_n)_{n\ge 1}$ is  $\Inc$-invariant, then it stabilizes in the sense that
\[
 \Inc(\mathcal{N}(\Delta_n))= \mathcal{N}(\Delta_{n+1})\quad \text{for }\ n\gg0,
\]
where $\mathcal{N}(\Delta_n)$ denotes the set of non-faces of $\Delta_n$ (see \cite[Theorem 3.6]{HS12}). One might wonder whether an analogous conclusion holds for {\cinv} chains of simplicial complexes.

\begin{problem}
 Let $(\Delta_n)_{n\ge 1}$ be a {\cinv} chain of simplicial complexes with $\Delta_n$ having the vertex set $[n]$. Is it true that the chain always stabilizes, that is,
\[
 \Inc(\Delta_n)=\Delta_{n+1}\quad\text{for all }\ n\gg0?
\]
\end{problem}

\end{document}